\documentclass[letterpaper, 10 pt, journal, twoside]{ieeetran}

\IEEEoverridecommandlockouts    
\usepackage[hyphens]{url}
\usepackage{hyperref}
\hypersetup{colorlinks=false,breaklinks=true}
\usepackage{amsfonts,amsthm,amssymb,mathrsfs}
\usepackage{amsmath,mathtools}
\usepackage{algorithmic}
\usepackage{booktabs,tabularx}
\usepackage{cases}
\usepackage[ruled,vlined]{algorithm2e}
\usepackage{array}
\usepackage{xcolor}
\usepackage{textcomp}
\usepackage{stfloats}
\usepackage{url}
\usepackage{verbatim}
\usepackage{graphicx}
\usepackage[utf8]{inputenc}
\usepackage[english]{babel}
\usepackage{array}
\usepackage{adjustbox}
\usepackage{cite}
\usepackage{xcolor}
\usepackage{tabularx}
\usepackage{enumitem}
\usepackage{marginnote}
\usepackage{multirow}
\def\BibTeX{{\rm B\kern-.05em{\sc i\kern-.025em b}\kern-.08em
    T\kern-.1667em\lower.7ex\hbox{E}\kern-.125emX}}
\usepackage{balance}
\usepackage{glossaries}

\graphicspath{{./figures/}}
\newtheorem{theorem}{Theorem}
\newtheorem{definition}{Definition}
\newtheorem{proposition}{Proposition}
\newtheorem{lemma}{Lemma}

\newtheorem{remark}{Remark}

\newtheorem{standing}{Standing Assumption}
\newcommand{\norm}[1]{\left\lVert#1\right\rVert}

\newcommand{\R}{\mathbb{R}}

\newcommand{\N}{\mathbb{N}}

\newcommand{\E}{\mathbb{E}}

\newcommand{\PP}{\mathbb{P}}

\newcommand{\mc}[1]{\mathcal{#1}}

\newcommand{\eqdef}{\coloneqq}

\newcommand{\prob}[2][]{
	\ifthenelse{
		\isempty{#1}}{\mathbb{#2}}{\mathbb{#2}\left\{#1\right\}
	}
}

\newcommand{\col}{\mathrm{col}}

\newcommand{\set}[2]{\left\{ #1\ \left| \ #2 \right. \right\}}

\newcommand{\bs}[1]{\boldsymbol{#1}}
\newcommand{\bsone}{\boldsymbol{1}}

\newcommand{\blue}[1]{{\leavevmode\color{blue}{#1}}}

\newcommand{\op}{\operatorname}
\newcommand{\normsq}[1]{\lVert#1\rVert^2}

\makeglossaries
\newacronym{iid}{i.i.d.}{independent and identically distributed}
\newacronym{wrt}{w.r.t.}{with respect to}
\newacronym{wlog}{w.l.o.g.}{without loss of generality}
\newacronym{PAC}{PAC}{probably approximately correct}
\newacronym{SNEP}{SNEP}{stochastic Nash equilibrium problem}
\newacronym{SNE}{SNE}{stochastic Nash equilibrium}
\newacronym{SGD}{SGD}{stochastic gradient descent}
\newacronym{NE}{NE}{Nash equilibrium}
\newacronym{GNE}{GNE}{generalized Nash equilibrium}
\newacronym{GNEP}{GNEP}{generalized Nash equilibrium problem}
\newacronym{SGNEP}{SGNEP}{stochastic generalized Nash equilibrium problem}
\newacronym{SGNE}{SGNE}{stochastic generalized Nash equilibrium}
\newacronym{v-SGNE}{v-SGNE}{variational equilibrium}
\newacronym{KKT}{KKT}{Karush–Kuhn–Tucker}
\newacronym{FB}{FB}{forward-backward}
\newacronym{FBF}{FBF}{forward-backward-forward}
\newacronym{EG}{EG}{extragradient}
\newacronym{BR}{BR}{best-response}
\newacronym{MSE}{MSE}{mean squared error}
\newacronym{MLE}{MLE}{maximum likelihood estimation}
\newglossaryentry{VI}
{
	name={VI},
	description={variational inequality},
	first={\glsentrydesc{VI} (\glsentrytext{VI})},
	plural={VIs},
	descriptionplural={variational inequalities},
	firstplural={\glsentrydescplural{VI} (VIs)}
}
\newacronym{DR}{DR}{Douglas-Rachford}
\newacronym{SAA}{SAA}{sample average approximation}
\newacronym{SA}{SA}{stochastic approximation}
\newacronym{PEV}{PEV}{plug-in electric vehicle}
\newacronym{DSO}{DSO}{distribution system operator}
\newacronym{EV}{EV}{electric vehicle}
\newacronym{LS}{LS}{least-squares}
\newglossaryentry{SVI}
{
	name={SVI},
	description={stochastic variational inequality},
	first={\glsentrydesc{SVI} (\glsentrytext{SVI})},
	plural={SVIs},
	descriptionplural={stochastic variational inequalities},
	firstplural={\glsentrydescplural{SVI} (SVIs)}
}

\begin{document}

\title{Actively learning equilibria in Nash games\\ with misleading information}
\author{Barbara Franci, Filippo Fabiani and Alberto Bemporad
\thanks{This work was partially supported by the European Research Council (ERC), Advanced Research Grant COMPACT (Grant Agreement No. 101141351).
B. Franci is with the Department of Mathematical Sciences, Politecnico di Torino, Corso Duca degli Abruzzi 24, 10129, Turin, italy ({\tt \footnotesize barbara.franci@polito.it}). F. Fabiani and A. Bemporad are with the IMT School for Advanced Studies Lucca, Piazza San Francesco 19, 55100, Lucca, Italy 
({\tt \{filippo.fabiani, alberto.bemporad\}@imtlucca.it}). 
}
}

% \markboth{Journal of \LaTeX\ Class Files,~Vol.~18, No.~9, September~2020}%
% {How to Use the IEEEtran \LaTeX \ Templates}

\maketitle         
\thispagestyle{empty}
\pagestyle{empty}

\begin{abstract}
	We develop a scheme based on active learning to compute equilibria in a \gls{GNEP}. Specifically, an external observer (or entity), with little knowledge on the multi-agent process at hand, collects sensible data by probing the agents' \gls{BR} mappings, which are then used to recursively update local parametric estimates of these mappings.
	Unlike \cite{fabiani2024active}, we consider the realistic case in which the agents share corrupted information with the external entity for, e.g., protecting their privacy. Inspired by a popular approach in stochastic optimization, we endow the external observer with an inexact proximal scheme for updating the local \gls{BR} proxies. This technique will prove key to establishing the convergence of our scheme under standard assumptions, thereby enabling the external observer to predict an equilibrium strategy even when relying on masked information.
\end{abstract}


\glsresetall

\section{Introduction}
\label{sec:intro}
%In the rapidly changing energy grid, a \gls{DSO} has to face the challenge of managing the flexibilities offered by the increasing amount of smart home appliances or \glspl{EV}. A \gls{DSO} has in fact to design energy prices according to the end-users consumption profile, which however might not be known in advance, especially because the end-users themselves should broadcast their consumption plan accurately \cite{corradi2012controlling,vespermann2017offering,bilenne2022privacy}.
%It is then crucial to develop tools to accomplish such a prediction task under limited knowledge about the decision process driving users' price-responsiveness. 
%
%
%As an example, consider a population of \gls{EV} on the distribution grid \cite{bilenne2022privacy,cenedese2019charging}, where
%every user wants to determine an optimal \gls{EV} charging schedule to minimize the energy costs in a so-called \gls{GNEP} \cite{facchinei2007}. The solution to this problem (a Nash equilibrium) is heavily dependent on the prices affecting the private cost function of the agents. Therefore, an appropriate design of the prices, based on accurate predictions from the \gls{DSO} is more than necessary to guarantee an efficient and convenient use of the distribution grid, even more so, if the process is stochastic, i.e., affected by some uncertainty or noise.

\IEEEPARstart{P}{redicting} a possible outcome in problems involving self-interested and privacy-preserving agents is a key requirement for their indirect control. As a prominent example, a \gls{DSO} ideally wishes to exploit the flexibility offered by the widely spread smart-home appliances and \glspl{EV} for an efficient usage of the distribution grid. To this end, a \gls{DSO} typically designs energy prices to induce a certain collective consumption profile of the end-users, which can be predicted in advance only if these users are willing to share sensitive information \cite{corradi2012controlling,bilenne2022privacy}.

Akin to \cite{fabiani2024active}, in this paper we take the perspective of an external observer interested in learning a so-called \gls{GNE} for a population of selfish agents taking part to a \gls{GNEP}. Given its little knowledge on the multi-agent process at hand, such an external observer is only allowed for querying the \gls{BR} mappings held by the agents. The latter, however, may be intentionally reluctant to reveal private information exactly, could erratically change their mind when presented with the same scenarios, may optimize their individual objectives with scarce accuracy, or the communication channels might be imperfect. For these reasons, we assume that the information passed to the external observer is masked by noise, as may happen in economic models \cite{renault2004communication}, competitive versions of multi-agent feedback controller synthesis \cite{crusius1999sufficient}, or signal processing \cite{mochaourab2012robust}. We then design an active learning-based scheme \cite{Set12,Bem23} for the external entity that, despite the misleading information collected, allows to predict a \gls{GNE} via faithful approximations of the agents' \glspl{BR}.

Learning an equilibrium strategy from a centralized perspective based on noisy information has been considered in that branch of literature denoted as {\it simulation-based game theory}.
Several works \cite{viqueira2019learning,areyan2020improved,marchesi2020learning} indeed proposed different schemes to approximate the original matrix games and associated equilibria by leveraging noisy samples of agents' costs provided by an oracle. Existing techniques addressing simulated matrix games with finite decision sets include also stochastic \cite{vorobeychik2008stochastic} and sample-average approximation \cite{vorobeychik2010probabilistic}, as well as methods based on Bayesian optimization \cite{al2018approximating,picheny2019bayesian}. While the former analyze the asymptotic properties of equilibria obtained from simulation-based models, also attaching probabilistic certificates on their approximation quality, the latter leverage statistical modeling tools acting as emulators of the agents' costs. Tailored acquisition functions for equilibrium learning are then designed based on the resulting posterior distributions.

In contrast, we design an active learning procedure for an external entity that  iteratively makes suitable queries to estimate the \gls{BR} mappings held by the agents, aiming at an exact prediction of a \gls{GNE} for the \gls{GNEP} in which they take part (\S \ref{sec:problem_formulation}). To deal with a possibly misleading information provided by the agents, we take inspiration from a popular approach in stochastic optimization to let the external observer update the local \gls{BR} proxies with noisy data by means of an inexact proximal scheme. This will prove to be a key tool for learning a \gls{GNE}, as well as to accompany the overall scheme with convergence guarantees under common assumptions. 

Our main contributions can be summarized as follows:
%Although the interaction among agents is not particularly relevant for our approach, as long as the decision of each agent is affected by the decision of any other agent. We focus instead on the external entity, which iteratively makes queries to estimate private action-reaction mappings to predict the stationary profile \emph{exactly}
%We assume however that when they communicate with the external entity, some uncertainty affects this communication. In particular, we consider a case in which some noise is added for privacy preserving reason, or because some agents re intentionally malicious or simply because of faulty communication. the presence of this noise is the main difference compares to \cite{fabiani2024active}
\begin{enumerate}
	\item[i)] We propose a stochastic variant of the active learning scheme derived in \cite{fabiani2024active} (\S \ref{sec:active_learning_scheme}). Our iterative algorithm is based on an inexact proximal update to learn the parameters approximating the \gls{BR} mappings of the agents;
	\item[ii)] Under standard assumptions \cite{lei2020synchronous,facchinei201012}, we show how these parameters can be learned exactly. Besides improving the results of \cite{fabiani2024active}, where such a condition was identified as sufficient for the convergence of the overall scheme and only verified ex-post, it is instrumental for proving that the external entity can asymptotically predict a \gls{GNE} of the underlying \gls{GNEP} (\S \ref{sec:convergence_analysis}).
\end{enumerate}
We finally discuss practical implementation details, which are then used to test our algorithm on a numerical case study involving the indirect control of a population of \glspl{EV} that tries to optimize the collective day-ahead charging schedule (\S \ref{sec:details_simulations}).

%\subsection{Notation} 
%\textcolor{blue}{Barbara: revise if we need all this notation}
\subsubsection*{Notation}
$\N$, $\R$ and $\R_{\geq 0}$ denote the set of natural, real, and nonnegative real numbers, respectively. $\N_0 \eqdef \N \cup \{0\}$.
% identifies the set of extended real numbers.
%The transpose of a matrix $A \in \R^{n \times n}$ is $A^\top$, while
%$A \succ 0$ ($\succcurlyeq 0$) denotes its positive (semi)definiteness. 
%For a vector $v \in \R^n$ and a matrix $A \succ 0$, $\|v\|_2$ denotes the standard Euclidean norm, while $\|\cdot\|_A$ the $A$--induced norm $\|v\|_A \coloneqq \sqrt{v^\top A v}$.
For a vector $v \in \R^n$, $\|v\|_2$ denotes the standard Euclidean norm.
%, where $\inner{\cdot}{\cdot} : \mathbb{R}^n \times \mathbb{R}^n \rightarrow \mathbb{R}$ stands for the standard inner product. 
%$\mc{B}_\theta(\bar x)$ represents 
%The $n$-dimensional ball centered around $\bar x$ with radius $\theta > 0$ is $\mc{B}_\theta(\bar x) \eqdef \set{x \in \R^n}{\|x - \bar x\|_2 \leq \theta}$.
%$I_{n}$, $\bsone_n$, and $\bs{0}_n$ denote the $n \times n$ identity matrix, the vector of all $1$, and $0$, respectively (we omit the dimension $n$ whenever clear). Given a set $\mc X \subseteq \R^n$, $\iota_\mc X:\R^n\to\bar\R$ denote the associated indicator function, i.e., $\iota_\mc X(x)=0$ if $x\in \mc X$, $\iota_\mc X(x)=+\infty$ otherwise. 
The operator $\col(\cdot)$ stacks its arguments in column vectors or matrices of compatible dimensions.
For example, given vectors $x_1,\dots,x_N$ with $x_i\in\mathbb{R}^{n_i}$ and $\mc I=\{1,\dots,N \}$, we denote $\bs{x} \eqdef (x_1 ^\top,\dots ,x_N^\top )^\top = \col((x_i)_{i\in\mc I}) \in \R^n$, $n \eqdef \sum_{i\in \mc I} n_i$, and $ \bs{x}_{-i} \eqdef \col(( x_j )_{j\in\mc I\setminus \{i\}})$, where $(\cdot)^\top$ denotes the transpose. Abusing notation, we also use $\bs x = (x_i, \bs x_{-i})$. 
%In general, $\bs x$ denotes the collective decision variable of all agents, $\bs z$ indicates the noisy information communicated by the agents and $\bs y$ represents a generic vector of dimension $n$.
%We define the filtration $\mc F=\{\mc F_t\}_{t\in\N}$, i.e., a family of $\sigma$-algebras with $\mathcal{F}_{0} = \sigma\left(X_{0}\right)$ and $\mathcal{F}_{t} = \sigma\left(X_{0}, \eta_{1}, \eta_{2}, \ldots, \eta_{t}\right)$ for all $t \geq 1$, such that $\mc F_t\subseteq\mc F_{t+1}$ for all $t\in\N$. In words, $\mc F_t$ contains the information up to iteration index $t$.
%The uniform distribution on the interval $[a,b]$ is denoted by $\mathcal U(a,b)$. 
With $\E_{\PP}[\bs z]=\op{col}(\E_{\PP_i}(z_i))_{i\in\mc I}$ we consider the stacked vector $\bs z$ and apply the expected value component-wise. The uniform distribution on $[a,b]$ is denoted by $\mc U(a,b)$, while the normal distribution with mean $\mu$ and variance $\sigma^2$ by $\mc N(\mu,\sigma^2)$. 

%%%%%%%%%%%%%%%%%%%%%%%%%%%%%%%%%%%%%%%%%%%%%%%%%%%%%%%%%%%%%%%%%%%%%%%
\section{Problem formulation}\label{sec:problem_formulation}
%We consider here a stochastic version of the problem tackled in \cite{fabiani2024active}, where an external observer aims at predicting a possible outcome in decision-making problems consisting of $N$ agents, indexed by $\mc I \coloneqq \{1, \ldots, N\}$. The agents mutually influence each other and their decision policies are kept private.

A \gls{GNEP} involves $N$ self-interested agents, indexed by the set $\mc I\eqdef\{1,\dots,N\}$, where each of them controls a decision variable $x_i \in \R^{n_i}$. Their aim is to minimize a local cost function $J_i : \R^n \to \R$, $n\eqdef\sum_{i\in\mc I} n_i$, subject to both local and coupling constraints. As such, a \gls{GNEP} can be written as a collection of mutually-coupled optimization problems \cite{facchinei201012}:
\begin{equation}\label{eq_game}
	\forall i\in\mc I:\left\{
	\begin{aligned}
		&\underset{x_i \in \mc X_i}{\textrm{min}} & & J_i(x_i,\bs x_{-i})\\
		&~\textrm{ s.t. } & & (x_i, \bs x_{-i}) \in \Omega.
	\end{aligned}	
	\right.
\end{equation}

Thus, each cost function $J_i(x_i,\bs x_{-i})$ depends not only on the local variable $x_i$, but also on the decisions of the other agents, $\bs x_{-i}=\op{col}((x_j)_{j\in\mc I \setminus \{i\}})$.
For every agent $i\in\mc I$, $\mc X_i$ represents the set of local constraints, while the coupling constraint set is $\Omega\subseteq\R^n$. The collective feasible set of the \gls{GNEP} in \eqref{eq_game} is then given by 
$\Omega\cap\mc X$, with $\mc X\eqdef\prod_{i\in\mc I} \mc X_i$, and the feasible decision set for agent $i\in\mc I$, parametric in $\bs x_{-i}$, is
$\mc X_i(\bs x_{-i})=\set{x_i \in \mc X_i}{(x_i, \bs x_{-i}) \in \Omega}$.
A popular solution concept for a \gls{GNEP} is the so-called \gls{GNE}, defined next:
\begin{definition}\label{def_gne}
	A collective decision vector $\bs x^\star\in\Omega\cap\mc X$ is a \gls{GNE} of the \gls{GNEP} in \eqref{eq_game} if, for all $i \in \mc I$,
	%\begin{equation}\label{eq_NE}
	$J_i(x^\star_i,\bs x^\star_{-i})\leq \textnormal{\textrm{min}}_{x_i\in\mc X_i(\bs x^\star_{-i})}~J_i(x_i,\bs x^\star_{-i})$.
	%\end{equation}
	\hfill$\square$
\end{definition}
Roughly speaking, at a \gls{GNE}, none of the agents has incentive to deviate from the strategy currently taken. 
%\gls{SGNEP} are also often considered \cite{ravat2011characterization,franci2021stochastic}. 
%In this case, the cost functions are given in expectation, i.e., $J_i(x_i,\bs x_{-i})=\E[h_i(x_i,\bs x_{-i},\xi_i)]$, where $h_i:\R^n\times\R^d\to\R$ and $\xi_i\in\R^d$ is a random variable in the probability space $(\Xi, \mathcal{F}, \mathbb{P})$ with $\Xi=\Xi_1 \times$ $\cdots \times \Xi_N$.
%We assume that the expected value is well defined for all the feasible $\bs{x}\in\Omega$. 
%Hence, a \gls{SGNEP} is a collection of optimization problems as in \eqref{eq_game} and a \gls{SNE} can be defined, mutatis mutandis, as in Definition \ref{def_gne}. 
%For the sake of out problem, it is not relevant which of the two type og problems describes the interactions between the agents, as long as we can write the action-reaction mappings of the agents as
In the considered game-theoretic framework, a quantity of interest is represented by the agent's \gls{BR} mapping, formally defined as:
\begin{equation}\label{eq_ar_br}
	f_i(\bs x_{-i})\eqdef\underset{x_i\in\mc X(\bs x_{-i})}{\textrm{argmin}}~J_i(x_i,\bs x_{-i}).
\end{equation}
In words, each $f_i:\R^{n_{-i}}\rightrightarrows\R^{n_i}$, $n_{-i}\eqdef\sum_{j \in \mc I \setminus \{i\}} n_j$, expresses what is the best set of decisions agent $i$ can take, given the current decision of its opponents $\bs x_{-i}$. It is also instrumental to characterize a \gls{GNE}, since $\bs x^\star$ can be equivalently defined as a collective fixed point of the agents' \gls{BR} mappings, i.e., $x_i^\star\in f_i(\bs x_{-i}^\star)$, for all $i\in\mc I$.

% It follows then that a \gls{GNE} is a fixed point of the mapping in \eqref{eq_ar_br}, i.e., $x_i^\star=f_i(\bs x_{-i}^\star)$, hence, a stationary action profile according to Definition \ref{def:stat_action}. The external entity aims then at reconstructing the BR mapping of the agents to predict the \gls{GNE} of the game in \eqref{eq_game}.

While not particularly restrictive, the following conditions ensure the existence of at least a \gls{GNE} \cite{facchinei201012}:
\begin{standing}[\textup{\gls{BR} mappings and constraints}]\label{standing:standard}
	For all $i \in \mc I$, $f_i : \R^{n_{-i}} \to \R^{n_i}$ is single-valued and continuous, with each $x_i\mapsto J_i(x_i, \bs x_{-i})$ convex on $\mc X_i(\boldsymbol{x}_{-i})$. The collective feasible set $\Omega\cap\mc X \subseteq\R^n$ is nonempty, convex\blue{,} and bounded.
	\hfill$\square$
\end{standing}
%\blue{To ensure that a \gls{GNE} exists, we add the following \cite{facchinei201012}.
%\begin{standing}[Cost functions]
%For every agent $i$, the cost function $J_i(\cdot, \bs x_{-i})$ is convex on $\mc X_i(\boldsymbol{x}_{-i})$.
%\end{standing}}

In this framework, we assume that the external entity has no knowledge about the private
functions $J_i$ but can probe the agents' \gls{BR} mappings in order to collect data and predict an equilibrium strategy~$\bs x^\star$.
%Specifically, we assume that each agent controls a decision variable $x_i \in \R^{n_i}$ to react to the other agents' actions $\bs x_{-i} \in \R^{n_{-i}}$, $n_{-i} \eqdef \sum_{j \in \mc I \setminus \{i\}} n_j$ via a private action-reaction mapping $f_i : \R^{n_{-i}} \to \R^{n_i}$.
%Moreover, the whole population of agents has to meet collective constraints with the general form $\bs x \eqdef (x_i, \bs x_{-i}) \in \Omega \subseteq \R^n$, $n \eqdef n_i + n_{-i}$. The portion of constraints involving agent $i$ can be embedded in the action-reaction mapping $f_i(\cdot)$ together with local constraints. Therefore, for any given $\bs x_{-i}$, $f_i(\bs x_{-i})$ is such that $(f_i(\bs x_{-i}), \bs x_{-i}) \in \Omega$, for all $i \in \mc I$.
%The external entity has limited information on the overall decision-making process, particularly on the action-reaction mappings privately held by the agents. Moreover, 
Unlike \cite{fabiani2024active}, however, we assume here that instead of communicating the exact \gls{BR}, each agent shares a noisy information $z_i=\tilde f_i(\bs x_{-i},\eta_i)$ with the external entity. Specifically,
%\ffcomment{Why $\tilde f$? What's its definition? In any case (here and after), shouldn't be $\tilde f_i$?}
%\bfcomment{$\tilde f$ is defined in Standing Assumption \ref{ass_zeta}. Yes, there should be the subscript $i$ - I think I got them all.}
$\eta_i:\Xi_i\to\R^d$ is a random vector defined on the probability space $(\Xi_i,\mc F_i,\PP_i)$ with \emph{unknown} probability distribution to all parties involved. 

%\ffcomment{Perhaps we can make $\eta_i$ a local random vector, i.e., defined on $(\Xi_i,\mc F_i,\PP_i)$, or it is necessary that all $\eta_i$'s have the same probability distribution?}
%\bfcomment{Yes}

As commonly assumed in a stochastic framework \cite{lei2020synchronous,lei2022distributed}, we postulate next a condition on the bias associated to $z_i$:
\begin{standing}[\textup{Unbiased noisy information}]\label{ass_zeta}
For all $i\in\mc I$ and $\bs x_{-i}\in\R^{n_{-i}}$, it holds that $\E_{\PP_i}[z_i]=x_i$, i.e., $\E_{\PP_i}[\tilde f_i(\bs x_{-i},\eta_i)]=f_i(\bs x_{-i})$. 
\hfill$\square$
\end{standing}

Considering noisy \glspl{BR} provides a more practical setting in which the agents, intentionally or unintentionally,
do not share exact best responses, for some of the reasons we described in \S \ref{sec:intro}. However,
%, flavor to the problem addressed, since agents may be reluctant to sharing private information, being them uncertain about it or simply for protecting their privacy. \blue{In addition, it might reflect the case of imperfect communication channels connecting agents and central entity.
Standing Assumption \ref{ass_zeta} postulates, realistically, that the agents have no interest in boycotting the central entity with its prediction task, i.e., they are not intentionally malicious: it is each agent's own interest to achieve
an agreement with the other agents.

The external entity, equipped with some learning procedure $\mathscr{L}$, shall then predict a \gls{GNE} by leveraging possibly misleading, yet non-private, information.
%\begin{definition}[\textup{Stationary action profile}]\label{def:stat_action}
%	A collective action profile $\bs{x}^\star \in \Omega$ is said to be \emph{stationary} if, for all $i \in \mc I$, $x^\star_i =  f_i(\bs x^\star_{-i})$.
%	\hfill$\square$
%\end{definition}
%In words, a stationary profile is a fixed point of the action-reaction mappings. Therefore, in this situation, none of the agents has incentive to deviate from the action currently taken.
%To this end, it is endowed with some learning procedure $\mathscr{L}$, which exploits non-private data (i.e., the noisy decisions $z_i$) collected iteratively from the agents through a query process.
Specifically, let us consider an estimate $\hat f_i : \R^{n_{-i}} \times \R^{p_i} \to \R^{n_i}$ of the $i$-th \gls{BR} mapping $f_i(\cdot)$. This \gls{BR} proxy is parametrized by $\theta_i \in \Theta_i\subseteq\R^{p_i}$, a quantity that shall be updated iteratively by integrating the data obtained from the agents through a smart query process, which will be described in the next section. 
\begin{standing}[\textup{Parameter set and \gls{BR} proxies}]\label{ass_param}
For all $i\in\mc I$, it holds that:
\begin{enumerate}
	\item[(i)] $\Theta_i$ is a closed, compact, and convex set;
	\item[(ii)] The mapping $\theta_i \mapsto \hat f_i(\bs x_{-i}, \theta_i)$ is continuous.
	\hfill$\square$
\end{enumerate}
\end{standing}
While not postulated in \cite{fabiani2024active}, in our stochastic framework we need Standing Assumption~\ref{ass_param}.(i) to restrict the set of parameters, thereby ensuring that the learning procedure can compensate for the noise. 
%Moreover, we allow for generic surrogate mappings $\hat f_i$ instead of affine ones. 
This will be key to establishing the asymptotic convergence of the parameters, thus improving over \cite{fabiani2024active}, where this condition was identified as sufficient for concluding on the convergence of the overall procedure.

\begin{algorithm*}[!t]
	\caption{Active learning-based method with misleading information}\label{alg:learning_embedded_BR}
	%	\DontPrintSemicolon
	%	\SetArgSty{}
	%	\SetKwFor{ForAll}{for all}{do}{end forall}
	\smallskip
	\textbf{Initialization:} $\bs x^0 \in \Omega$, $\theta_i^0 \in \R^{p_i} \text{ for all } i \in \mc I$\\
	\smallskip
	\textbf{Iteration $(k \in \N_0)$:} \\
	\vspace{.1cm}
	\hspace{.5cm}$\bullet$ External entity computes
	%\begin{itemize}
		%\item[$\bullet$] $\theta_i^{k+1} \in  \underset{\xi_i \in \R^{p_i}}{\textrm{argmin}} \; \frac{1}{k} 
		%\sum_{t = 1}^k \ell_i(z_i^t, \hat f_i(\hat{\bs x}^t_{-i}, \xi_i)), \forall i \in \mc I$
		\begin{equation}\label{eq_inexact_br}
		%\theta_i^{k+1} \in \set{\xi\in\Theta_i}{\E[\normsq{\xi-\hat\theta_i(z_i^k,\bs x_{-i}^k,\theta_i^k)}|\mc F_k]\leq(\alpha_i^k)^2 \text{ a.s.}}
		\theta_i^{k+1} \in \set{\xi\in\Theta_i}{\E_{\PP_i}[\normsq{\xi-\hat\theta_i(z_i^k,\hat{\bs x}_{-i}^k,\theta_i^k)}|\mc F^k]\leq(\alpha_i^k)^2 \text{ a.s.}} \text{ for all } i \in \mc I
		\end{equation}
		%\smallskip
		%\item[$\bullet$] 
		
		\vspace{.1cm}
		\hspace{.5cm}$\bullet$ External entity defines $\mc M(\theta^{k+1}) \eqdef 	\textrm{argmin}_{\bs x \in \Omega\cap\mc X} \; \sum_{i\in \mc I} \left\|x_i - \hat f_i(\bs x_{-i}, \theta^{k+1}_i) \right\|^2_2$, and computes
		\begin{equation}\label{eq_min_norm}
		\hat{\bs x}^{k+1}  = \underset{\bs x \in \R^n}{\textrm{argmin}}\left\{\tfrac{1}{2}\|\bs x\|_2^2~\textrm{ s.t. }~\bs x \in \mc M(\theta^{k+1})\right\} 
		\end{equation}
%	\end{itemize}
%	For all $i\in\mc I$ agent $i$ computes:
%	\begin{itemize}
%		\item[$\circ$] $x^{k+1}_i = f_i(\hat{\bs x}^{k+1}_{-i})$
%	\end{itemize}
%	Each agent $i\in\mc I$ shares
	
	\vspace{.1cm}
	\hspace{.5cm}$\bullet$ External entity collects corrupted \glspl{BR}, for all $i\in\mc I$:
	%\begin{itemize}
		%\item[$\circ$] 
		$$z^{k+1}_i = \tilde f_i(\hat{\bs x}^{k+1}_{-i},\eta_i^{k+1})$$
	%\end{itemize}
\end{algorithm*}

%\ffcomment{Referring to Algorithm~\ref{alg:learning_embedded_BR}, I think we should assume something like $(z^k_i,\hat{\bs x}^k_{-i})\in\Omega$ (or maybe $(\E_{\PP}[z_i^k],\hat{\bs x}^k_{-i}) \in \Omega$), which is not granted for free in this noisy version.}
%\bfcomment{$(\E_{\PP}[z_i^k],\hat{\bs x}^k_{-i}) \in \Omega$ already holds by Standing Assumption \ref{ass_zeta} and by definition of $\mc M$.
%
%Is this necessary?}

%\textcolor{blue}{Barbara: do we want to add some possible applications here?}

\section{Active learning with misleading information}\label{sec:active_learning_scheme}
%The main instructions of our scheme based on active learning are reported in 
The proposed active \gls{GNE} learning scheme is summarized in
Algorithm~\ref{alg:learning_embedded_BR}.
%, where the black-filled bullets refer to the tasks the external observer is required to perform, while the empty bullet to the one performed by the agents in $\mc I$. 
Note that in the last step the agents act as oracles, i.e., they provide samples consisting of noisy \glspl{BR} that the external observer uses for learning.
Specifically, at every iteration $k$ the external entity integrates samples just collected to perform an inexact update of the \gls{BR} proxies as in \eqref{eq_inexact_br}. In fact, given $\bs \eta^t=\op{col}((\eta_i^t)_{i\in\mc I})$ at the generic $t$-th iteration, $\mc F^k$ is there defined according to the filtration $\mc F=\{\mc F^k\}_{k\in\N}$, i.e., the family of $\sigma$-algebras with $\mathcal{F}^{0} = \sigma\left(X^{0}\right)$ and $\mathcal{F}^{k} = \sigma\left(X^{0}, \bs\eta^{1}, \bs\eta^{2}, \ldots, \bs\eta^{k}\right)$ such that $\mc F^k\subseteq\mc F^{k+1}$ for all $k\in\N$. In words, $\mc F^k$ contains the information up to iteration $k$. 

If each probability distribution $\mathbb{P}_i$ was known, the external entity would ideally implement the following proximal rule:
\begin{equation}\label{eq_exact_br}
	\hat\theta_i(z_i,\hat{\bs x}_{-i},\theta_i) \in \underset{\xi_i \in \Theta_i}{\textrm{argmin}} \left\{L_i(\xi_i|z_i,\hat{\bs x}_{-i})+\frac{\mu}{2}\normsq{\xi_i-\theta_i}_2\right\},
\end{equation} 
for all $i \in \mc I$, $\mu>0$, where in particular%$w=(x_i,\bs x_{-i},\theta_i)$ and
\begin{equation}\label{eq_L}
L_i(\theta_i | z_i,\hat{\bs x}_{-i})=\E_{\PP_i}[\ell_i(z_i, \hat f_i(\hat{\bs x}_{-i}, \theta_i))].
\end{equation}
However, since each $\E_{\PP_i}$ is unavailable, we propose for the external entity to focus on \eqref{eq_inexact_br} as a viable option.

The loss function $L_i : \Theta_i \to \R$ measures the dissimilarity between the information received via $z_i$, and the estimate $\hat f_i$. Note that $L_i$ does not depend explicitly on $x_i$ and $\bs x_{-i}$, since those are quantities provided through samples. In addition, $\ell_i : \Theta_i\times\Xi_i \to \R$ depends on $\eta_i$ via $z_i$, hence the expected value \gls{wrt} $\PP_i$. We then impose what follows:

%We keep them in the notation as a reminder and to ease the understanding of the following notation:

%\blue{Barbara: would it be more clear if we write $L_i(\xi_i|x_i,\bs x_{-i})$?}
%\ffcomment{I think it is.}

\begin{standing}[\textup{Training loss function}]\label{ass_diff}
For all $i\in\mc I$, the following conditions hold true:
\begin{enumerate}
	\item[(i)] The mapping $\theta_i\mapsto L_i(\theta_i|x_i,\bs x_{-i})$ is convex and twice continuously differentiable;
	\item[(ii)] The mapping $\theta_i\mapsto \ell_i(z_i, \hat f_i(\hat{\bs x}_{-i}, \theta_i))$ is differentiable;
	\item[(iii)] For all $(x_i, \bs x_{-i}, \theta_i) \in \Omega \cap \mc X \times \Theta_i$, $0 \le  L_i(\theta_i|x_i,\bs x_{-i}) < \infty$, with $ L_i(\theta_i|x_i,\bs x_{-i})= 0 \iff x_i = \hat f_i(\bs x_{-i}, \theta_i)$.
	\hfill$\square$
\end{enumerate}
\end{standing}
While Standing Assumption \ref{ass_diff} actually turns the inclusion in \eqref{eq_exact_br} into equality, the following condition, postulated also in \cite{lei2020synchronous, lei2022distributed}, will be key for our convergence analysis:
\begin{standing}[\textup{Proximal map}]\label{ass_contractive}
%The proximal mapping $\hat\theta_i(\cdot,\cdot,\theta_i)$ in \eqref{eq_exact_br} is $a$-contractive, $a\in(0,1)$, i.e., for all $(z_i, \blue{\hat{\bs x}_{-i}}, \theta_i),(z'_i,\hat{ \bs x}'_{-i}, \theta'_i) \in \Omega \cap \mc X \times \Theta_i$, $\|\hat\theta(z_i,\hat{\bs x}_{-i},\theta_i)-\hat\theta(z'_i,\blue{\hat{\bs x}'_{-i}},\theta'_i)\|\leq a\|(z_i,\hat{\bs x}_{-i},\theta_i)-(x'_i,\blue{\hat{\bs x}'_{-i}},\theta'_i)\|$.
The proximal mapping $\hat\theta_i(\cdot,\cdot,\theta_i)$ in \eqref{eq_exact_br} is $a$-contractive, $a\in(0,1)$, i.e., for all $\theta_i, \theta'_i \in  \Theta_i$, $\|\hat\theta(\cdot,\cdot,\theta_i)-\hat\theta(\cdot,\cdot,\theta'_i)\|\leq a\|\theta_i-\theta'_i\|$.
\hfill$\square$
\end{standing}
We postulate this property as an assumption, although sufficient conditions guaranteeing the contractivity of $\hat\theta_i(\cdot,\cdot,\theta_i)$ can be obtained similarly to \cite[Prop.~12.17]{facchinei201012} and \cite[\S 2.2]{lei2020synchronous}. However, since the external entity does not know the probability distribution $\PP_i$ of the noise, the expected value in \eqref{eq_exact_br}--\eqref{eq_L} can not be computed exactly. 
This is the reason why, inspired  by \cite{lei2020synchronous}, we propose an inexact scheme as described in \eqref{eq_inexact_br}. In fact, such an instruction is asymptotically equivalent to the exact proximal mapping in \eqref{eq_exact_br}, but it can be computed through the iterations. The parameters $\alpha_i^k$ in \eqref{eq_inexact_br}, instead, form a deterministic sequence that meets the following conditions:
\begin{standing}[\textup{Accuracy sequence}]\label{ass_step}
For all $i\in\mc I$, the sequence $\{\alpha_i^k\}_{k\in\N}$  is such that $\sum_{k\in\N_0}\alpha_i^k<\infty$ and, for all $k\in\N_0$, $\alpha_i^k\geq0$  and $\lim_{k\to\infty}\alpha_i^k=0$.
\hfill$\square$
\end{standing}

\begin{remark}\label{rem:inexact_proximal}
In \cite{lei2020synchronous}, a stochastic approximation method is used to obtain a solution to \eqref{eq_inexact_br} which is $\alpha_i^k$-close to the exact one of \eqref{eq_exact_br}. As a consequence of Standing Assumption \ref{ass_step}, convergence to the exact solution holds (see Lemma \ref{lemma:local_exactness}). This consists in performing a number of stochastic proximal gradient descent steps, proportional to the outer iteration index $k$ of Algorithm \ref{alg:learning_embedded_BR} \cite[\S 3.4]{lei2020synchronous}.
%, enough to obtain a solution that is $\alpha_i^k$-close to the exact one \cite[\S 3.4]{lei2020synchronous}. 
In particular, at iteration $k$, for all $i\in\mc I$, the external entity performs the following steps for $t>0$:
\begin{equation}\label{eq:proximal_gradient}
		\begin{aligned}
		\xi_i^{t+1}=\op{proj}_{\Theta_i}(\xi_i^t-\gamma^t&(\tfrac{1}{S}\textstyle\sum_{j=1}^S\nabla_{\theta_i}\ell_i(z_i^{k,j},\hat f_i(\hat{\bs x}_{-i}^k, \xi^t_i))\\
		&+\mu(\xi_i^t-\theta_i^k))),
	\end{aligned}
\end{equation}
with $\gamma^t$ being a vanishing step-size sequence and $\{z_i^{k,j}\}_{j=1}^S$ being a collection of $S$ samples of the noisy queries. The iterative procedure stops, 
%when an inexact solution with accuracy $\alpha_i^k$ is reached, 
say after $\bar t$ iterations, and sets $\theta_i^{k+1}=\xi_i^{\bar t}$.
In this case, some further assumption on the expected-valued gradients should be considered---see, e.g., \cite[Ass.~1.(c), 1.(d)]{lei2020synchronous}.
Other algorithms can be however used and integrated with different approximation schemes.
\hfill$\square$
\end{remark}
%\textcolor{blue}{Barbara: do we want more details? Maybe for/in the simulation section?}
%Referring to Algorithm~\ref{alg:learning_embedded_BR}, we preliminarily define, for all $i \in \mc I$, $\textrm{argmin}_{\xi_i \in \R^{p_i}} \; \frac{1}{k} \sum_{t = 1}^k \ell_i(x_i^t, \hat f_i(\hat{\bs x}^t_{-i}, \xi_i))$. 

%\FFnote{Perchè introduci tutti e due gli aggiornamenti di $\theta_i^{k+1}$ sopra?}

By leveraging the \gls{BR} surrogates updated through \eqref{eq_inexact_br}, the external entity then designs the next query point $\hat{\bs x}^{k+1}$ to collect new information from the agents according to \eqref{eq_min_norm}, i.e., as the minimum norm strategy profile falling into the set:
\begin{equation}\label{eq:minimizers}
	\mc M(\theta^{k+1}) \eqdef 	\underset{\bs x \in \Omega\cap\mc X}{\textrm{argmin}} \; \sum_{i\in \mc I} \left\|x_i - \hat f_i(\bs x_{-i}, \theta^{k+1}_i) \right\|^2_2,
\end{equation}
where $\mc M:\R^p\rightrightarrows\Omega$, $p \eqdef \sum_{i\in \mc I} p_i$. This set contains all collective profiles that are the closest to a fixed point of each $\hat f_i(\cdot, \theta_i^{k+1})$, i.e., closest to a \gls{GNE} as defined in Definition~\ref{def_gne} and discussion following \eqref{eq_ar_br}. Indeed, if each $\hat f_i(\bar{\bs x}_{-i}, \bar\theta^k_i)$ was exactly equal to $f_i(\bar{\bs x}_{-i})$, and the minimum in~\eqref{eq:minimizers} was identically zero, then $\bar{\bs x}\in\mc M(\bar\theta^k)$ would be a \gls{GNE} of the \gls{GNEP} in \eqref{eq_game}. Such a smart selection of the query points amounts to the ``active'' part of Algorithm~1, and allows the central entity to accumulate (noisy) information in a neighborhood of a point that is the closest to a true \gls{GNE}. This will be key for the technical analysis carried out in the next section.
In \eqref{eq:minimizers}, $\theta^{k+1}$ represents the whole collection of parameters $\{\theta_i^{k+1}\}_{i\in\mc I}$ characterizing the estimate mappings, which at every iteration coincides with the argument of the corresponding parameter-to-query mapping $\mc M(\cdot)$. It then follows from the definition of $\mc M$ and from Standing Assumption \ref{ass_zeta} that $(\E_\PP[z_i],\hat{\bs x}_{-i})\in\Omega\cap\mc X$. 
%An exact knowledge of $\Omega\cap\mc X$ is not fundamental for computing \eqref{eq:minimizers}, since the external entity may always leverage a conservative approximation of the collective feasible set $\Omega\cap\mc X$, which may then be iteratively refined by collecting new samples.
Once obtained the minimum norm vector $\hat{\bs x}^{k+1}$ in \eqref{eq_min_norm}, the external entity queries each agent with $\hat{\bs x}_{-i}$, which in turn reacts through a noisy \gls{BR} $z^{k+1}_i = \tilde f_i(\hat{\bs x}^{k+1}_{-i},\eta_i^{k+1})$. The observer finally collects all these data, and the process repeats.

\section{Convergence analysis}\label{sec:convergence_analysis}
%\red{FF: At the moment, it seems we are considering generic $\hat f_i$. I suspect, however, we need further assumptions---please, give a look at \cite[\S V.C]{fabiani2024active}.}\\
Before studying the asymptotic properties of the active learning procedure in Algorithm~\ref{alg:learning_embedded_BR}, we postulate some assumptions on the learning procedure $\mathscr{L}$. We then prove some preliminary results, functional to the asymptotic analysis.

%To ensure that we can use a generic surrogate mapping $\hat f_i$ and not an affine one as in \cite{fabiani2024active}.
%\begin{standing}
%	For all $i\in\mc I$, for all $\bs x_{-i} \in \mathbb{R}^{n_{-i}}, \theta_i \mapsto \hat f_i(\bs x_{-i}, \theta_i)$ is continuous.
%\end{standing}

%\begin{standing}[\textup{Training loss}]\label{standing:learning_procedure}
%For all $i \in \mc I$ and for all $(x_i, \bs x_{-i}, \theta_i) \in \Omega \cap \mc X \times \Theta_i$, $0 \le L_i(x_i, \bs x_{-i},\theta_i) < \infty$, with $L_i(x_i, \bs x_{-i}, \theta_i)= 0$ if and only if $x_i = \hat f_i(\bs x_{-i}, \theta_i)$.
%	\hfill$\square$
%\end{standing}

%\FFnote{Whether or not you need $\E[\ell_i(y_i, \hat f_i(\bs y_{-i}, \theta_i))] = 0$ is unclear -- it should depend on which between \eqref{eq:param_update_1} and \eqref{eq:param_update_2} one decides to use, and hence from the proofs of the associated auxiliary results. With $y_i = \hat f_i(\bs y_{-i}, \theta_i)$, one is requiring to match the true BR exactly I guess, which may make sense...}

%The conditions in Standing Assumption~\ref{standing:learning_procedure} allow us to characterize the choice of the loss functions adopted in \eqref{eq:param_update_2}. To this end, a common choice is to use the \gls{MSE} loss:
%\begin{equation}
%	\ell_i(z_i^t, \hat f_i(\hat{\bs x}^t_{-i}, \xi_i)) =\|z_i^t-\hat f_i(\hat{\bs x}^t_{-i}, \xi_i)\|_2^2,
%	\label{eq:MSE-loss}
%\end{equation}
%which make \eqref{eq:param_update_2} a linear \gls{LS} problem. 

In particular, our analysis will be based on the possibility of matching pointwise the \gls{BR} mapping $f_i$ of each agent. %, a key property possessed by the affine surrogates in \eqref{eq:linear-predictor}, which can hence be used with no restrictions. 
To this aim, for all $i \in \mc I$ and for all $(x_i, \bs x_{-i}) \in \Omega\cap\mc X$, let
$$\mc A_i(x_i, \bs x_{-i}) =\{\tilde\theta_i\in\Theta_i|L_i(\tilde\theta_i|x_i, \bs x_{-i}) = 0\}.$$

%Moreover, in view of the specific structure \eqref{eq:linear-predictor} we consider, the affine proxies are also continuous \gls{wrt} their second argument, i.e., $\hat f_i(\bs y_{-i}, \cdot)$ is continuous for all $\bs y_{-i} \in \R^{n_{-i}}$.

This set is instrumental to prove the following crucial result: 

%\red{FF: How is $\theta_i^\star$ defined below? Can we just use $\bar\theta_i$ instead?}

\begin{lemma}\label{lemma:local_exactness}
	%For all $i \in \mc I$, let $\{(z_i^t,\hat{\bs x}^t_{-i})\}_{t = 1}^k$ be a collection of samples to be employed in the update rule for the parameter $\theta_i$ in Algorithm \ref{alg:learning_embedded_BR}, where each $(\E_{\PP}[z_i^t],\hat{\bs x}^t_{-i}) \in \Omega$. 
	For all $i \in \mc I$, let $\{\theta_i^k\}_{k\in\N}$ be the sequence generated by \eqref{eq_inexact_br} in Algorithm \ref{alg:learning_embedded_BR}.
	If $\lim_{k \to \infty} \E_{\PP_i}[z_i^k] = \bar{x}_i$, $\lim_{k \to \infty} \hat{\bs x}^k_{-i} = \bar{\bs x}_{-i}$ so that $(\bar{x}_i, \bar{\bs x}_{-i})\in\Omega\cap\mc X$, then for all $i\in\mc I$, $\lim_{k \to \infty} \theta_i^k = \bar\theta_i$, and $\lim_{k\to\infty}\E_{\PP_i}\left[\norm{\theta_i^k-\bar\theta_i}\right]=0$ a.s.. Moreover, $\bar\theta_i \in \mc A_i(\bar{x}_i, \bar{\bs x}_{-i})$, i.e., \eqref{eq_inexact_br} converges to a solution of the exact proximal scheme in \eqref{eq_exact_br}.
	\hfill$\square$
\end{lemma}
\begin{proof}
The fact that $\lim_{k \to \infty} \theta_i^k = \bar\theta_i$ for all $i \in \mc I$ follows analogously to \cite[Prop.~1]{lei2020synchronous} 
by using contractivity (Standing Assumption~\ref{ass_contractive}), the unbiased noise (Standing Assumption~\ref{ass_zeta}) and convexity of $\Theta_i$ (Standing Assumption~\ref{ass_param}) on $\|\theta_i^{k+1}-\bar\theta_i\|$, together with the vanishing property of $\{\alpha_i^k\}$ (Standing Assumption~\ref{ass_step}). With the same assumptions, it follows similarly that $\lim_{k\to\infty}\E_{\PP_i}\left[\norm{\theta_i^k-\bar\theta^k}\right]=0$ \cite[Prop.~2.(b)]{lei2020synchronous}.
%The facts that $\lim_{k \to \infty} \theta_i^k = \bar\theta^k$ and $\lim_{k\to\infty}\E_{\PP_i}\left[\norm{\theta_i^k-\bar\theta^k}\right]=0$ for all $i \in \mc I$ follow the same steps as those of \cite[Prop.~1 and 2]{lei2020synchronous} by comparing the inexact mapping \eqref{eq_inexact_br} with the exact contractive one \eqref{eq_exact_br}. We recall that $L_i$ does not depend on $x_i$ and $\bs x_{-i}$ explicitly, since those quantities are made available through (noisy) samples. Therefore, the main difference with \cite{lei2020synchronous} is that the contractive property reads as in Standing Assumption \ref{ass_contractive}. 
The last statement, instead, follows analogously to \cite[Lemma 4.2]{fabiani2024active} from the properties of $\ell_i$ (Standing Assumption \ref{ass_diff}).
\end{proof}

%\red{FF: What does ``optimal $\bar\theta_i$'' mean here?}

%\red{FF: What does ``the main difference with \cite{lei2020synchronous} is that the contractive property reads as in Standing Assumption \ref{ass_contractive}'' mean? Perhaps we should give more details rather than ``follow the same steps as those of...''. Also, what is $\tilde\theta_i$ below? How is it related with $\theta_i^\star$?}

Lemma~\ref{lemma:local_exactness} says that when all the ingredients involved in Algorithm~\ref{alg:learning_embedded_BR} converge, then the pointwise approximation of $f_i$ shall be exact
at ${\bs x}_{-i}$, i.e., $\hat f_i({\bs x}_{-i},\bar\theta_i)=\E_{\PP_i}[\tilde f_i(\bs x_{-i},\eta_i)]=f_i({\bs x}_{-i})$, with $\bar\theta_i\in\mc A_i(x_i,\bs{x}_{-i})$. We will then prove in Proposition~\ref{th:convergence} that the conditions required in Lemma~\ref{lemma:local_exactness} are verified.
To simplify the notation, let
$r(\bs x,\theta)=\sum_{i\in \mc I} \|x_i - \hat f_i(\bs x_{-i}, \theta_i) \|^2_2.$
Next, we impose some conditions on $r(\bs x,\theta)$:
\begin{standing}\label{ass_min}
The following conditions hold true:
\begin{itemize}
\item[(i)] For all $x \in \Omega, \theta \mapsto r(\bs x, \theta)$ is convex and differentiable;
\item[(ii)] For all $\theta \in \mathbb{R}^p, \bs x \mapsto r(\bs x, \theta)$ is convex and continuous;
\item[(iii)] For all $\theta \in \mathbb{R}^p$, the vector $\partial r(\bs x, \theta) / \partial \theta \in \mathbb{R}^p$ of partial derivatives is bounded \gls{wrt} $\bs x$. \hfill$\square$
\end{itemize}
\end{standing}
The following technical result characterizes the properties of the sequence of query points $\{\hat{\bs x}^k\}_{k\in\N}$ produced by the central entity in the second step \eqref{eq_min_norm} of Algorithm~\ref{alg:learning_embedded_BR}.
\begin{proposition}%[Proposition 1 \cite{fabiani2024active}]
	\label{prop:f}
%	Let $\{\theta^k\}_{k\in\N}$ be a sequence
%	so that $\lim_{k \to \infty} \theta_i^k = \tilde{\theta}_i$ a.s. for all $i \in \mc I$, and 
	Let $\mc M(\tilde\theta)=\{\tilde{\bs x}\}$. Then the
	sequence $\{\hat{\bs x}^k\}_{k\in\N}$ generated by \eqref{eq_min_norm} is feasible, i.e., $\hat{\bs x}^k\in\Omega\cap\mc X$ for all $k\in\N$, and satisfies
	$
	\lim_{k\rightarrow\infty}\hat{\bs x}^k=\tilde{\bs x}.
	$
	\hfill$\square$
\end{proposition}
\begin{proof}
In view of Lemma \ref{lemma:local_exactness}, $\lim_{k \to \infty} \theta_i^k = \tilde{\theta}_i$ a.s. for all $i \in \mc I$. The proof then follows from Standing Assumptions \ref{standing:standard} and \ref{ass_min} which imply that $\mc M(\theta^k)$ is a convex set \cite[Lemma 3.3]{fabiani2024active}. From the same assumptions then follows that the sequence $\{\hat{\bs x}^k\}_{k\in\N}$ is bounded and its cluster points $\bar{\bs x}$ belong to $\mc M(\tilde\theta)$ \cite[Lemma 3.6, 3.7]{fabiani2024active}. By contradiction we can instead show that it can not happen that $\bar{\bs x}\neq\tilde{\bs x}$ \cite[Lemma 3.7]{fabiani2024active}.
\end{proof}
We are now ready to state the asymptotic properties of the active learning-based technique summarized in Algorithm~\ref{alg:learning_embedded_BR}:
%\red{FF: What is the $\PP$ below and how is it related to each $\PP_i$?}
\begin{theorem}\label{th:convergence}
%	Let $\{\theta^k\}_{k\in\N}$ be a sequence
%	so that $\lim_{k \to \infty} \theta_i^k = \tilde{\theta}_i$ a.s. for all $i \in \mc I$, and 
	Let $\mc M(\tilde \theta) = \{\tilde{\bs x}\}$.
	Then, $\lim_{k \to \infty} \|\E_{\PP}[\bs z^k-\hat{\bs x}^k]\|_2 = 0$, and the sequences $\{\bs x^k\}_{k \in \N}$ and $\{\hat{\bs x}^k\}_{k \in \N}$ generated by Algorithm~\ref{alg:learning_embedded_BR} converge to the same \gls{GNE} of the \gls{GNEP} in \eqref{eq_game}.
	\hfill$\square$
\end{theorem}
\begin{proof}
It follows from Standing Assumptions \ref{standing:standard}, \ref{ass_diff} and \ref{ass_min} by noting that, in view of the consistency property in Lemma~\ref{lemma:local_exactness}, and by Proposition \ref{prop:f} the pointwise approximation shall be exact, namely each $\tilde{\theta}_i$ is so that, for all $i\in\mc I$, $\| \hat f_i(\tilde{\bs x}_{-i},\tilde{\theta}_i) - \E_{\PP_i}[\tilde f_i(\tilde{\bs x}_{-i},\eta_i)]\|_2 = 0$ \cite[Th.~4.5]{fabiani2024active}.
%	Since \eqref{eq:minimizers} is solved at the global optimum at every $k \in \N$, and the single-valuedness of each $f_i(\cdot)$, this latter relation readily yields $\sum_{i\in \mc I} \|\tilde x_i - \hat f_i(\tilde{ \bs x}_{-i}, \tilde \theta_i)\|^2_2 = \sum_{i\in \mc I} \|\tilde x_i - f_i(\tilde{ \bs x}_{-i})\|^2_2 = 0 = \sum_{i\in \mc I} \|\tilde x_i - \bar x_i\|^2_2$, and hence $\|\tilde x_i - \bar x_i\|_2 = 0$ for all $i \in \mc I$. This means that $\bar{ \bs x} = \tilde{ \bs x}$, and in view of Definition~\ref{def:stat_action}, they coincide with a fixed point of the action-reaction mappings $f_i(\cdot)$, as  $\sum_{i\in \mc I} \|f_i(\bar{\bs{x}}_{-i}) - \bar x_i\|^2_2=0$, i.e., $\|f_i(\bar{\bs{x}}_{-i}) - \bar x_i\|_2=0$ for all $i\in\mc I$, concluding the proof.
\end{proof}
Theorem \ref{th:convergence} establishes that the external entity achieves convergence to the true values, i.e., it predicts both an exact \gls{GNE} of the game and the \gls{BR} mappings, despite the possibly misleading information passed by the agents.
%\begin{remark}
%	Note that Standing Assumption~\ref{standing:standard} is not sufficient to guarantee the existence of a \gls{GNE} of the game in \eqref{eq_game}. As a consequence of Theorem \ref{th:convergence}, the \gls{GNEP} at hand then admits at least one according to Definition \ref{def_gne} \cite[Cor.~4.6]{fabiani2024active}.
%	\hfill$\square$
%\end{remark}
%\red{FF: Are the only conditions imposed on the game, i.e., Standing Assumption~\ref{standing:standard} sufficient for proving that an equilibrium exists? Otherwise, akin to \cite{fabiani2024active}, it should be stressed that the convergence of the algorithm guarantees the existence of a \gls{GNE}.}

%%%%%%%%%%%%%%%%%%%%%%%%%%%%%%%%%%%%%%%%%%%%%%%%%%%%%%%%%%%
\section{Implementation details and simulation results}\label{sec:details_simulations}
We now discuss several implementation details related to Algorithm~\ref{alg:learning_embedded_BR} that will be employed to perform numerical experiments on a charging coordination problem for \glspl{EV}.

\subsection{Practical considerations}\label{subsec:implementation_detail}
A distinct feature of the proposed active learning-based scheme is represented by the inexact proximal step in \eqref{eq_inexact_br}, which can be accomplished as discussed in Remark~\ref{rem:inexact_proximal}. To this end, performing for instance the stochastic proximal gradient descent in \eqref{eq:proximal_gradient} requires one the availability of a batch of $S$ samples $\{z_i^{k,j}\}_{j=1}^S$ at every outer iteration $k$. The latter can be obtained by the central entity either probing the $i$-th \gls{BR} mapping $S$ times with the same $\hat{\bs x}^{k-1}_{-i}$, or producing synthetic samples. While the former may not represent a viable approach in a realistic case involving, e.g., human agents, the latter can be always pursued on the basis of the data collected up to iteration $k$, i.e., $\{z_i^j\}_{j=1}^k$. Among the simplest approaches, a \gls{MLE} method \cite{bishop2006pattern} allows one to estimate the (possibly time-varying) measurement noise covariance matrix $R_i^k$ using measurement residuals (innovations), i.e.,
$
 e_i^k=z_i^k- \hat f_i(\bs z^k_{-i}, \theta^k_i).
$
The likelihood function associated to the measurements $\{z_i^j\}_{j=1}^k$ given $R^k$ is then:
$
	\mc L_i(R_i^k) = \prod_{j=1}^{k} \exp \left(-\frac{1}{2} (e_i^j)^\top (P_i^j)^{-1} e_i^j \right)/\sqrt{|2\pi P_i^j|},
$
where $P_i^j$ denotes the measurement covariance at the $j$-th outer iteration. Taking the logarithm of the likelihood function, we obtain
$
	\log \mc L_i(R_i^k) = -\frac{k}{2} \log |P^k_i| - \frac{1}{2} \sum_{j=1}^{k} (e^j_i)^\top (P^j_i)^{-1} e^j_i,
$
which, in case the measurements are affected by Gaussian noise, allows one to estimate $R^k_i$ by maximizing $\log \mc L_i(R_i^k)$ \gls{wrt} $R_i^k$ itself. Thus, setting the derivative of the above to zero and solving for $R_i^k$ yields:
$
	\hat{R}^k_i = \frac{1}{k} \sum_{j=1}^{k} e^j_i (e^j_i)^\top,
$
which is the sample covariance of the residuals, and can be employed to produce synthetic samples $\{z_i^{k,j}\}_{j=1}^S$ for \eqref{eq:proximal_gradient} through, e.g., multivariate normal sampling. Specifically, one generates data $z_i^{k,j}=z_i^k+V_i^k \nu_i^j$, where $\nu_i^j\sim\mc N(0,I_{n_i})$ and $V_i^k$ is a matrix obtained from the Cholesky or singular value decomposition of $\hat{R}^k_i$. We will later exploit this empirical approach in \S \ref{sec:case study}. % to test Algorithm~\ref{alg:learning_embedded_BR} on a numerical case study.

Note that the convergence property of our active learning-based scheme requires only a pointwise exact approximation of the \gls{BR} mappings held by the agents for the external observer to successfully accomplish the prediction task, despite noisy data. For this reason, as observed in \cite{fabiani2024active} it is convenient for the external entity to adopt affine \gls{BR} proxies $\hat f_i(\cdot,\theta_i)$, i.e.,
\begin{equation}\label{eq:affine_proxies}
	\hat f_i(\bs x_{-i}, \theta_i)=\Lambda_i\begin{bmatrix}\bs x_{-i}\\1\end{bmatrix},
\end{equation}
for $\Lambda_i \in \R^{n_i \times (n_{-i}+1)}$---note that $\theta_i$ is the vectorization of $\Lambda_i$, with $p_i=n_i(n_{-i}+1)$. In case one adopts a standard \gls{MSE} for the training, i.e., $\ell_i(z_i,\hat f_i(\bs x_{-i}, \theta_i))=\tfrac{1}{2}\normsq{z_i-\Lambda_i\left[\bs x_{-i}^\top ~~ 1\right]^\top}$ such a design choice allows to automatically satisfy Standing Assumption~\ref{ass_diff}, as well as the requirements in Standing Assumption~\ref{ass_min}. Besides all these technical motivations, affine \gls{BR} surrogates also yield important practical consequences. Specifically, each gradient in \eqref{eq:proximal_gradient} reads as $(\Lambda^t_i [(\hat{\bs x}_{-i}^k)^\top ~~ 1]^\top - z_i^{k,j} ) [(\hat{\bs x}_{-i}^k)^\top ~~ 1]$, while solving \eqref{eq:minimizers} turns out to be a constrained \gls{LS} problem, which is convex in view of Standing Assumption~\ref{standing:standard}.

\begin{table}[tb]
	\caption{Indirect control of \glspl{EV} -- Simulation parameters}
	\label{tab:sim_val}
	\centering
	\begin{tabular}{lll}
		\toprule
		Parameters  &  Description   & Value \\
		\midrule
		$T$ & Time interval & $14$\\
		$N$ & Number of \glspl{EV}  & $10$\\
		$Q_i=q_i I_T$& Degradation cost -- quadratic term & $q_i\sim\mc U(0.006,0.01)$\\
		$c_i$& Degradation cost -- affine term & $\sim\mc U(0.055,0.095)^T$\\
		$d$& Normalized inflexibility demand & from \cite[Fig.~1]{ma2011decentralized}\\
		$\rho_i$ & Local charging requirement &$\sim\mc U(1.2,1.8)$\\
		$\bar c_i$ & Upper bound - power injection &$0.25$\\
		$\bar c$ & Grid capacity &$0.2$\\
		$a$ & Inverse price elasticity of demand & $0.8$\\
		$b$ & Baseline price & $0.02$\\
		\midrule
		$\eta_i$ & Additive noise on each \gls{BR} & $\sim \mc N(0,0.1)$\\
		$\Theta_i$ & Parameters' set &$[-10,10]^{p_i}$\\
		$K$ & Number of iterations (Alg.~\ref{alg:learning_embedded_BR}) &$200$\\
		$\gamma^t$ & Step-size in \eqref{eq:proximal_gradient} &$10^{-3t}$\\
		$\mu$ & Proximal parameter &$10$\\
		$\bar t$ & Iterations performed in \eqref{eq:proximal_gradient}  &$10k$\\
		\bottomrule
	\end{tabular}
\end{table}
\subsection{Case study: Indirect control of smart grids}\label{sec:case study}
We test our technique on an indirect control problem faced by \glspl{DSO}, which design price signals enabling the energy flexibility offered by price-sensitive end-users \cite{d2022exploiting}.

In particular, we consider a set of $N$ \glspl{EV} populating a distribution grid \cite{ma2011decentralized,cenedese2019charging}, where every selfish agent aims at determining an optimal \gls{EV} charging schedule over a certain discrete time interval $\{1,\ldots,T\}$ by controlling the energy injection $x_i \in \R_{\ge 0}^T$. The underlying problem is typically modeled as a \gls{GNEP}, consisting of the following collection of mutually coupled optimization problems: 
\begin{equation}\label{eq:single_prob_EV}
	\forall i \in \mc I : \left\{
	\begin{aligned}
		& \underset{x_i}{\textrm{min}} &&  \|x_i\|^2_{Q_i}+c_i^\top x_i + (a (\sigma(\bs x) + d) + b \bsone_T)^\top x_i\\
		& \textrm{ s.t. } && \bsone_T^\top x_i \ge \rho_i,~x_i \in [0, \bar x_i]^T,~\sigma(\bs x) \le \bar c.
	\end{aligned}
	\right.
\end{equation}

Each private cost function is composed of two terms: $\|x_i\|^2_{Q_i}+c_i^\top x_i$, which models the battery degradation cost, and $(a (\sigma(\bs x) + d) + b \bsone_T)^\top x_i$, which is associated to the electricity pricing. Here, $\sigma(\bs x)$ denotes the aggregate demand of the whole population of \glspl{EV},  defined as $\sigma(\bs x) = \tfrac{1}{N} \sum_{i=1}^{N} x_i \in \R^{T}_{\ge0}$, where $a>0$ represents the inverse of the price elasticity of demand, $b >0$ the baseline price, and $d\in\R^T_{\ge 0}$ the normalized average inflexible demand. In addition, each user has to satisfy both local and shared constraints due for instance to a minimum charging amount over the interval, $\bsone_T^\top x_i \ge \rho_i \geq0$, a cap on the power injection $x_i \in [0, \bar x_i]^T$, or accounting for intrinsic grid limitations, i.e., $\sigma(\bs x)+d \in [0, \bar c]^T$. 

In this framework, an equilibrium strategy $\bs x^\star$, which produces the aggregate consumption $\tfrac{1}{N} \sum_{i=1}^{N} x^\star_i$, heavily depends on the values of $a$ and $b$.
%, since they affect the cost function of each agent, namely $\bs x^\star=\bs x^\star(a,b)$. 
It is then clear how a suitable design of $a$ and $b$, based on an accurate prediction of the resulting $\bs x^\star(a,b)$, allows for an efficient usage of the distribution grid. Thus, a \gls{DSO} is interested in making accurate forecasts on the aggregate electricity consumption of end-users in response to price-signals, aimed at enabling flexibility offered by the users themselves. On the other hand, the smart query process proposed in \cite{fabiani2024active} does not account for the possible malice of end-users, who may not be willing to provide correct information, are uncertain or even contradictory about it.

%\begin{figure}[!t]
%	\centering
%	\includegraphics[width=.95\columnwidth]{S_impact}
%	\caption{Relative distance between the \gls{GNE} computed by relying on misleading information through Algorithm~\ref{alg:learning_embedded_BR} (i.e., $\bs x^{200}$), and the one obtained with noiseless data (i.e., $\bs x^\star$), averaged over $20$ numerical instances of \eqref{eq:single_prob_EV}.}
%	\label{fig:S_impact}
%\end{figure}
%\begin{figure}[!t]
%	\centering
%	\includegraphics[width=.95\columnwidth]{ev_convergence_comparison}
%	\caption{Relative distance sequence produced by Algorithm~\ref{alg:learning_embedded_BR} (solid blue line) and by \cite[Alg.~1]{fabiani2024active} exploiting noisy data (solid red line), averaged over $20$ numerical instances. \blue{The shaded colored areas represent the standard deviation over the different numerical trials, while} the shaded black region corresponds to the random initialization of both procedures.}
%	\label{fig:ev_convergence_comparison}
%\end{figure}

We conduct numerical experiments by using the values reported in Tab.~\ref{tab:sim_val}. Specifically, we assume the \gls{DSO} endowed with affine \gls{BR} proxies as in \eqref{eq:affine_proxies}, and additive noise affecting the agents' \glspl{BR}, i.e., $z_i=f_i(\bs x_{-i})+\eta_i$ for all $i\in\mc I$. While Algorithm~\ref{alg:learning_embedded_BR} is initialized as described in \cite[\S VI.A]{fabiani2024active}, we exploit the procedure in Remark~\ref{rem:inexact_proximal} for solving \eqref{eq_exact_br} with increasing accuracy at every outer iteration $k$. With this regard, we preliminary analyze the impact that the size of $S$ has on the computation of a \gls{GNE} with misleading information. For each agent, to generate synthetic samples $\{z_i^{k,j}\}_{j=1}^S$ at each iteration, we have adopted the \gls{MLE}-based approach discussed in \S \ref{subsec:implementation_detail}. Then, for each $S\in\{1,5,10,20,50\}$, we have generated $20$ numerical instances of \eqref{eq:single_prob_EV}, run \cite[Alg.~1]{fabiani2024active} with noise-free \gls{BR} samples for computing a reference \gls{GNE}, and then Algorithm~\ref{alg:learning_embedded_BR}. In Fig.~\ref{fig:S_impact}, which illustrates the box plot associated to the relative distance from a \gls{GNE} for each case, we observe that, as expected, a larger batch of samples allows for a better accuracy in the \gls{GNE} computation and reduces the related variance. On the other hand, we have also experienced a significant increase in the computational time, since each iteration of Algorithm~\ref{alg:learning_embedded_BR} with $S=1$ takes $2.79$[s] as worst-case average (i.e., with $k=200$), up to $52.8$[s] for $S=50$. Motivated by these considerations, we have then set $S=10$ and compared the query point sequences generated by Algorithm~\ref{alg:learning_embedded_BR} with a na\"ive implementation of \cite[Alg.~1]{fabiani2024active}. Also in this case, we have considered $20$ different numerical instances, with reference \gls{GNE} computed through \cite[Alg.~1]{fabiani2024active} by relying on noiseless data. From Fig.~\ref{fig:ev_convergence_comparison}, it is clear that, whether the procedure in Algorithm~\ref{alg:learning_embedded_BR} can cope with noisy \gls{BR} samples provided by the agents, running \cite[Alg.~1]{fabiani2024active} blindly with inexact data produces a non-convergent behavior.

\begin{figure}[!t]
	\centering
	\includegraphics[width=.95\columnwidth]{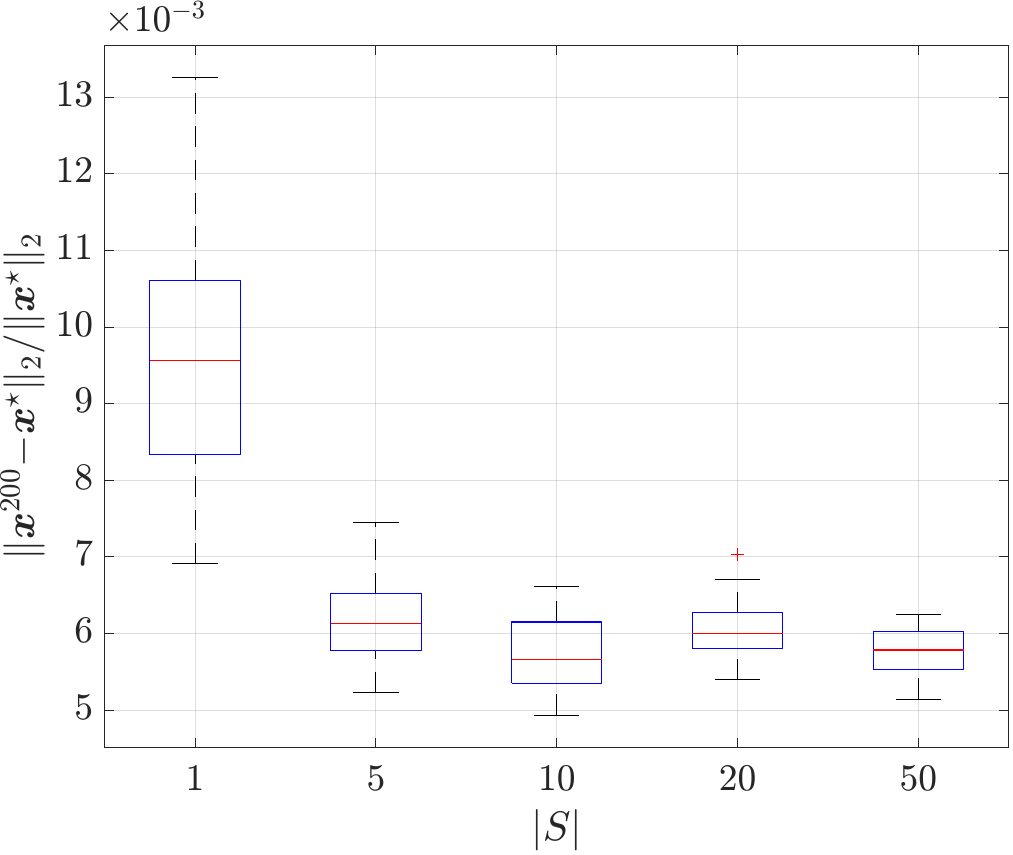}
	\caption{Relative distance between the \gls{GNE} computed by relying on misleading information through Algorithm~\ref{alg:learning_embedded_BR} (i.e., $\bs x^{200}$), and the one obtained with noiseless data (i.e., $\bs x^\star$), averaged over $20$ numerical instances of \eqref{eq:single_prob_EV}.}
	\label{fig:S_impact}
\end{figure}
\begin{figure}[!t]
	\centering
	\includegraphics[width=.95\columnwidth]{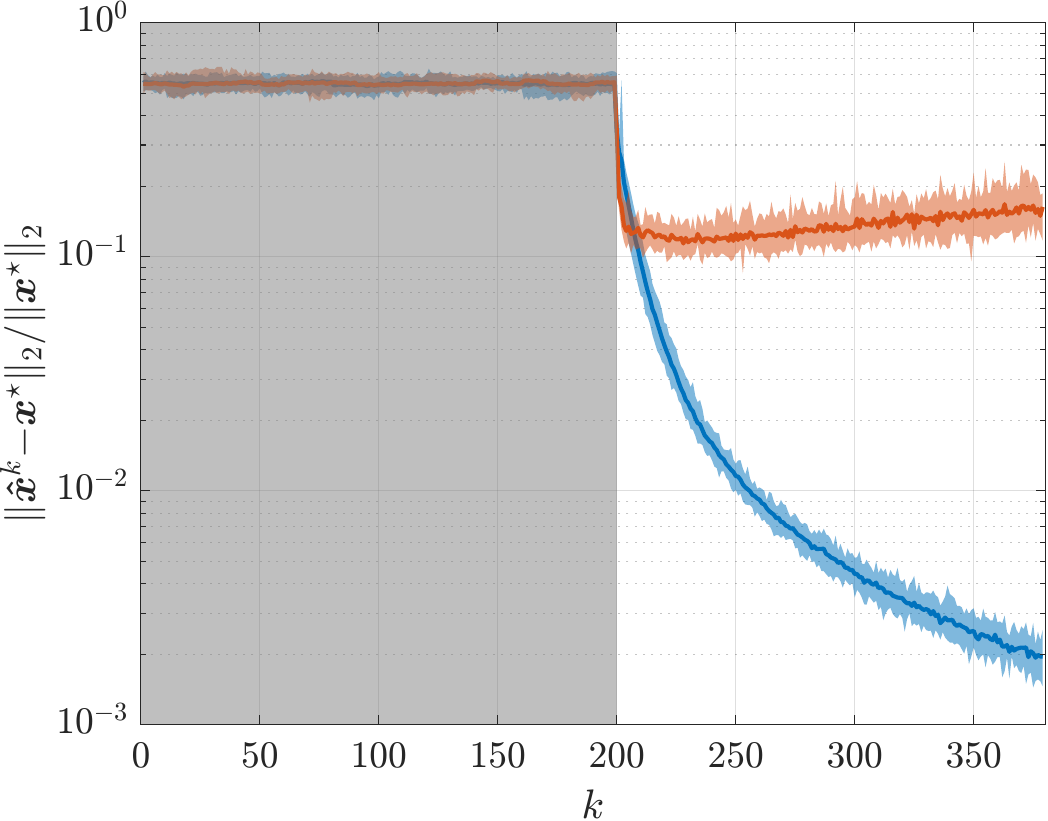}
	\caption{Relative distance sequence produced by Algorithm~\ref{alg:learning_embedded_BR} (solid blue line) and by \cite[Alg.~1]{fabiani2024active} exploiting noisy data (solid red line), averaged over $20$ numerical instances. The shaded colored areas represent the standard deviation over the different numerical trials, while the shaded black region corresponds to the random initialization of both procedures.}
	\label{fig:ev_convergence_comparison}
\end{figure}

%%%%%%%%%%%%%%%%%%%%%%%%%%%%%%%%%%%%%%%%
\section{Conclusion}
We have proposed a novel procedure based on active learning to let an external observer learn faithful local proxies of \gls{BR} mappings privately held by a population of agents taking part to a \gls{GNEP}. With the goal of predicting a \gls{GNE} of the underlying game, we have adopted an inexact proximal update of those surrogates that allows to integrate possible misleading information provided by the agents. We have shown that this technique guarantees the convergence of the \gls{BR} estimates and, at the same time, of the overall active learning scheme, ensuring that the external entity succeeds in its prediction task.

\bibliographystyle{IEEEtran}
\bibliography{stocBR_prediction}

\end{document}